\documentclass[12pt]{amsart}
\usepackage{a4}
\usepackage{amssymb}

\makeatletter
\@addtoreset{equation}{section}
\makeatother

\usepackage{color}
\usepackage[mathscr]{eucal}
\usepackage{amsmath,amsthm,amssymb}
\usepackage{mathrsfs}
\usepackage{enumerate}
\usepackage{bm}
\usepackage{graphicx}
\usepackage{verbatim}
\usepackage{wrapfig}
\usepackage{ascmac}
\usepackage{multicol}

\usepackage[dvipdfmx]{hyperref}

\newtheorem{Prop}{Proposition}[section]
\newtheorem{Thm}[Prop]{Theorem}
\newtheorem{Lem}[Prop]{Lemma}
\newtheorem{Rem}[Prop]{Remark}
\newtheorem{Ex}[Prop]{Example}

\theoremstyle{definition}
\newtheorem{Def}[Prop]{Definition}

\newcommand{\R}{{\mathbb R}}

\newcommand{\Z}{{\mathbb Z}}

\newcommand{\Map}{\mathrm{Map}}
\newcommand{\id}{\mathrm{id}}
\newcommand{\Aut}{\mathrm{Aut}}
\newcommand{\Fix}{\mathrm{Fix}}
\newcommand{\Inn}{\mathrm{Inn}}

\newcommand{\rnum}[1]{\expandafter{\romannumeral #1}}

\title{Flat connected finite quandles}

\author{Yoshitaka Ishihara} 
\address[Y.\ Ishihara]{Department of Mathematics, Hiroshima University, 
Higashi-Hiroshima 739-8526, Japan}

\author{Hiroshi Tamaru}
\address[H.\ Tamaru]{Department of Mathematics, Hiroshima University, 
Higashi-Hiroshima 739-8526, Japan}
\email{tamaru@math.sci.hiroshima-u.ac.jp}

\thanks{The second author was supported in part by KAKENHI (24654012, 26287012).} 

\date{}



\begin{document}

\maketitle

\begin{abstract}
Quandles can be regarded as generalizations of symmetric spaces. 
In the study of symmetric spaces, 
the notion of flatness plays an important role. 
In this paper, we define the notion of flat quandles, 
by referring to the theory of Riemannian symmetric spaces, 
and classify flat connected finite quandles. 
\end{abstract}

\section{Introduction}

Quandles were introduced by Joyce $($\cite{Joyce}$)$, and have been used for studying knot theory. A set $ X $ with a binary operator $ \ast : X \times X \rightarrow X $ is called a quandle if it satisfies three axioms, derived from the Reidemeister moves of a classical knot. It is worthwhile to mention that quandles are also related to symmetric spaces. Recall that a $($Riemannian$)$ manifold $ M $ is said to be symmetric if it is equipped with a globally defined ``symmetry'' $ s_{x} : M \rightarrow M $ for each $ x \in M $. It has been known in \cite{Joyce} that every symmetric space is a quandle, by defining the binary operator $ y \ast x := s_{x}(y) $. From this correspondence, several notions and ideas for symmetric spaces can be transferred to quandles. There have already been several studies of quandles from this viewpoint $($\cite{K.T.W.}, \cite{Tamaru}, \cite{Vendramin}, \cite{Wada}$)$. These studies can be regarded as an approach to the theory of ``discrete symmetric spaces'' $($\cite{Tamaru}$)$.

In this paper, we define the notion of flat quandles, by referring to the theory of Riemannian symmetric spaces. A Riemannian symmetric space is said to be flat if the curvature tensor vanishes identically. It is known that a Riemannian symmetric space is flat if and only if the group of displacements, the group generated by compositions $ s_{x} \circ s_{y} $ of two symmetries, is commutative $($\cite{Loos}$)$. Note that one can easily define the group of displacements for a quandle, and we define the flatness of a quandle by the commutativity of the group of displacements.

The main result of this paper is an explicit classification of flat connected finite quandles. We prove the following theorem:

\begin{Thm}\label{Thm1.1}
A quandle is flat, connected, and finite if and only if it is isomorphic to a direct product of dihedral quandles with odd prime power cardinalities.
\end{Thm}

Here the dihedral quandle with cardinality $ n $ is the set of $n$-equal dividing points on the unit circle $ S^{1} $, equipped with the natural symmetry. They can be regarded as a ``discrete $ S^{1} $''. Therefore, this theorem can be regarded as a discrete version of the following well-known fact: every flat connected compact Riemannian symmetric space is isomorphic to a flat torus, that is, a direct product of $ S^{1} $.

Note that the flatness is an important notion in the structure theory of Riemannian symmetric spaces. Therefore, flat quandles would be potentially useful for studying the structure theory of quandles, which will be studied in the forthcoming papers.

This paper is organized as follows. In Section~\ref{sec2}, we recall some necessary background on quandles, such as the dihedral quandles, connected quandles, and direct products of quandles. In Section~\ref{sec3}, we define quandle triplets, which correspond to homogeneous quandles. This notion is analogous to the notion of ``symmetric pairs'', and plays a fundamental role in the proof of Theorem~\ref{Thm1.1}. In Section~\ref{sec4}, we define the notion of flat quandles. We then study the structures of flat connected quandles in Section~\ref{sec5}, and flat connected finite quandles in Section~\ref{sec6}.

The authors would like to thank 
Takahiro Hashinaga, Kentaro Kimura, Akira Kubo, Takayuki Okuda, Yuichiro Taketomi, and Koshiro Wada 
for helpful comments.

\section{Preliminaries}\label{sec2}

In this section, we recall some necessary notions, such as the dihedral quandles, connected quandles, and direct products of quandles.

\subsection{Quandles and the dihedral quandles}

Quandles are usually defined by sets with binary operators satisfying three axioms, derived from the Reidemeister moves of a classical knot. However, we employ a formulation similar to symmetric spaces, as in \cite{Tamaru}. Let $ X $ be a set, and denote by
\begin{align*}
\Map(X,X) := \{ f : X \rightarrow X : \mbox{a map} \}.
\end{align*}

\begin{Def}\label{Def2.1}
Let $ X $ be a set, and consider a map
\begin{align*}
s: X \rightarrow \Map(X,X) : x \rightarrow s_{x}.
\end{align*}
Then the pair $ (X,s) $ is called a \textit{quandle} if
\begin{itemize}
\item[(Q1)]
For any $ x \in X $, $ s_{x}(x) = x $.
\item[(Q2)]
For any $ x \in X $, $ s_{x} $ is bijective.
\item[(Q3)]
For any $ x,y \in X $, $ s_{x} \circ s_{y} = s_{s_{x}(y)} \circ s_{x} $.
\end{itemize}
\end{Def}

For a quandle $ (X,s) $, the map $ s $ is called a \textit{quandle structure}. One knows that every connected Riemannian symmetric space is a quandle (\cite{Joyce}). Here we mention some other easy examples of quandles.

\begin{Ex}\label{Ex2.2}
The following $ (X,s) $ are quandles$:$
\begin{itemize}
\item[(1)]
The \textit{trivial quandle}$:$ $ X $ is any set and $ s_{x} := \id_{X} $ for any $ x \in X $.
\item[(2)]
The \textit{dihedral quandle of order $n$}$:$ $ X := \Z_{n} $ and, for any $ x,y \in X $,
\begin{align*}
s_{x}(y) := 2x-y.
\end{align*}
\end{itemize}
\end{Ex}

Here $ \Z_{n} $ is the cyclic group of order $ n $. In the remaining of this paper, we denote by $ R_{n} $ the dihedral quandle of order $ n $.

\begin{Def}\label{Def2.3}
Let $ (X,s^{X}) $ and $ (Y,s^{Y}) $ be quandles. Then $ f : X \rightarrow Y $ is called a \textit{homomorphism} if, for any $ x \in X $, it satisfies
\begin{align*}
f \circ s^{X}_{x} = s^{Y}_{f(x)} \circ f.
\end{align*}
\end{Def}

A bijective homomorphism is called an \textit{isomorphism}. If there exists an isomorphism between two quandles $ (X,s^{X}) $ and $ (Y,s^{Y}) $, then they are said to be \textit{isomorphic} and we write $ (X,s^{X}) \simeq (Y,s^{Y}) $.

\begin{Rem}\label{Rem2.4}
The dihedral quandle $ R_{n} $ is isomorphic to the set of $n$-equal dividing points on the unit circle $ S^{1} $, equipped with the natural symmetry. Recall that the natural symmetry $ s_{x} $ at $ x \in S^{1} \subset \R^{2} $ is the reflection with respect to the line $ \R x \subset \R^{2} $. Denote by $ R^{\prime}_{n} $ the set of $n$-equal dividing points on $ S^{1} $, and let us put numbers $ x_{0}, \ldots, x_{n-1} \in R^{\prime}_{n} $ clockwise $($or counterclockwise$)$. Then $ f(i) := x_{i} $ gives an isomorphism $ f : R_{n} \rightarrow R^{\prime}_{n} $.
\end{Rem}

\subsection{Homogeneous and connected quandles}

In this subsection, we recall the notions of homogeneous and connected quandles. One needs the automorphism groups and the inner automorphism groups of quandles.

An isomorphism from a quandle $ (X,s) $ onto $ (X,s) $ itself is called an \textit{automorphism} of $ (X,s) $. We denote the automorphism group of $ (X,s) $ by
\begin{align*}
\Aut(X,s) := \{ f : X \rightarrow X : \mbox{an automorphism} \}.
\end{align*}
It follows from (Q2) and (Q3) that $ s_{x} \in \Aut(X,s) $ for any $ x \in X $.

\begin{Def}\label{Def2.5}
Let $ (X,s) $ be a quandle. The group generated by $ \{ s_{x} \mid x \in X \} $ is called the \textit{inner automorphism group} of $ (X,s) $, and denoted by $ \Inn(X,s) $.
\end{Def}

\begin{Def}\label{Def2.6}
Let $ (X,s) $ be a quandle.
\begin{itemize}
\item[(1)]
$ (X,s) $ is said to be \textit{homogeneous} if $ \Aut(X,s) $ acts transitively on $ X $.
\item[(2)]
$ (X,s) $ is said to be \textit{connected} if $ \Inn(X,s) $ acts transitively on $ X $.
\end{itemize}
\end{Def}

One knows $ \Inn(X,s) \subset \Aut(X,s) $. Therefore, a connected quandle is homogeneous. Note that the converse does not hold in general. For example, the trivial quandle $ (X,s) $ is always homogeneous, but it is connected if and only if the cardinality of $ X $ is $ 1 $. The following gives another example.

\begin{Ex}\label{Ex2.7}
The dihedral quandle $ R_{n} $ is always homogeneous, but it is connected if and only if $ n $ is odd.
\end{Ex}

\subsection{The direct products of quandles}

In this subsection, we define the notion of direct products of quandles, and study their properties. Let $ M, M^{\prime}, N, N^{\prime} $ be sets. For maps $ f : M \rightarrow M^{\prime} $ and $ g : N \rightarrow N^{\prime} $, we define
\begin{align*}
f \times g : M \times N \rightarrow M^{\prime} \times N^{\prime} : (x,y) \mapsto (f(x),g(y)).
\end{align*}
One can easily show the following lemma.

\begin{Lem}\label{Lem2.8}
Let $ (X,s^{X}) $ and $ (Y,s^{Y}) $ be quandles. Then the pair $ (X \times Y , s^{X} \times s^{Y}) $ is also a quandle.
\end{Lem}

This quandle $ (X \times Y , s^{X} \times s^{Y}) $ is called the \textit{direct product of quandles $ (X,s^{X}) $ and $ (Y,s^{Y}) $}, and denoted by $ (X,s^{X}) \times (Y,s^{Y}) $. For the direct products of three quandles, one can easily see that 
\begin{align*}
((X,s^{X}) \times (Y,s^{Y})) \times (Z,s^{Z}) \simeq (X,s^{X}) \times ((Y,s^{Y}) \times (Z,s^{Z})).
\end{align*}
Hence we denote it by $ (X,s^{X}) \times (Y,s^{Y}) \times (Z,s^{Z}) $.

\begin{Prop}\label{Prop2.9}
Let $ (X,s^{X}) $ and $ (Y,s^{Y}) $ be quandles. Then we have
\begin{itemize}
\item[(1)]
$ \Inn((X,s^{X}) \times (Y,s^{Y})) \subset \Inn(X,s^{X}) \times \Inn(Y,s^{Y}) $.
\item[(2)]
The direct product $ (X,s^{X}) \times (Y,s^{Y}) $ is connected if and only if both of $ (X,s^{X}) $ and $ (Y,s^{Y}) $ are connected.
\end{itemize}
\end{Prop}

\begin{proof}
We show (1). Recall that $ \Inn((X,s^{X}) \times (Y,s^{Y})) $ is generated by
\begin{align*}
\{ (s^{X} \times s^{Y})_{(x,y)} \mid (x,y) \in X \times Y \}.
\end{align*}
Hence one can easily complete the proof of (1) by
\begin{align*}
(s^{X} \times s^{Y})_{(x,y)} = (s^{X}_{x} \times s^{Y}_{y}) \in \Inn(X,s^{X}) \times \Inn(Y,s^{Y}).
\end{align*}

We show (2). The ``only if'' part directly follows from (1). We show the ``if'' part. Assume that both of $ (X,s^{X}) $ and $ (Y,s^{Y}) $ are connected. Take any $ (x, y), (x^{\prime}, y^{\prime}) \in X \times Y $. Since $ (X,s^{X}) $ is connected, there exist $ x_{1}, \ldots, x_{n} \in X $ and $ \varepsilon_{1}, \ldots, \varepsilon_{n} \in \{ \pm1 \} $ such that
\begin{align*}
(s^{X}_{x_{1}})^{\varepsilon_{1}} \circ \cdots \circ (s^{X}_{x_{n}})^{\varepsilon_{n}}(x) = x^{\prime}.
\end{align*}
Then one has
\begin{align*}
f:= (s^{X} \times s^{Y})_{(x_{1},y)}^{\varepsilon_{1}} \circ \cdots \circ (s^{X} \times s^{Y})_{(x_{n},y)}^{\varepsilon_{n}} \in \Inn((X,s^{X}) \times (Y,s^{Y})),
\end{align*}
which satisfies $ f(x,y) = (x^{\prime},y) $. Similarly, since $ (Y,s^{Y}) $ is connected, there exists $ g \in \Inn((X,s^{X}) \times (Y,s^{Y})) $ such that $ g(x^{\prime},y) = (x^{\prime},y^{\prime}) $. This shows that $ \Inn((X,s^{X}) \times (Y,s^{Y})) $ acts transitively on $ X \times Y $.
\end{proof}

\section{The quandle triplets}\label{sec3}

In this section, we define and study quandle triplets. Note that the idea of quandle triplets has been introduced by Joyce (\cite{Joyce}, Section 7). Here we shall reformulate it in a similar way to ``symmetric pairs'', which play a fundamental role in the study of symmetric spaces.

\subsection{Quandle triplets and homogeneous quandles}

In this subsection, we define quandle triplets, and give a correspondence to homogeneous quandles. Let $ G $ be a group, and $ e $ be the unit element of $ G $. For a map $ \sigma : G \rightarrow G $, we denote by
\begin{align*}
\Fix(\sigma,G) := \{ g \in G \mid \sigma(g) = g \}.
\end{align*}
We also denote by $ \Aut(G) $ the automorphism group of $ G $ as a group.

\begin{Def}\label{Def3.1}
Let $ K $ be a subgroup of $ G $ and $ \sigma \in \Aut(G) $. Then $ (G,K,\sigma) $ is called a \textit{quandle triplet} if $ K \subset \Fix(\sigma,G) $.
\end{Def}

We give a correspondence between quandle triplets and homogeneous quandles. First of all, we construct a homogeneous quandle from a quandle triplet. Note that the following has essentially been known by Joyce (\cite{Joyce}, Section 7).

\begin{Prop}\label{Prop3.2}
Let $ (G,K,\sigma) $ be a quandle triplet. We define $ s $ as follows$:$
\begin{align*}
s_{[g]}([h]) := [g\sigma(g^{-1}h)] \quad ([g],[h] \in G/K).
\end{align*}
Then we have
\begin{itemize}
\item[(1)]
$ s : G/K \rightarrow \Map(G/K,G/K) $ is well-defined.
\item[(2)]
$ s $ is a quandle structure on $ G/K $.
\item[(3)]
$ (G/K,s) $ is a homogeneous quandle.
\end{itemize}
\end{Prop}

\begin{proof}
Let $ (G,K,\sigma) $ be a quandle triplet. Then one can directly prove (1) by using $ K \subset \Fix(\sigma,G) $.

We show (2). Conditions (Q1) and (Q3) can be checked directly. One can show (Q2) as follows: for each $ [g] \in G/K $, the map given by
\begin{align*}
s^{-1}_{[g]}([h]) := [g \sigma^{-1}(g^{-1}h)] \quad ([h] \in G/K)
\end{align*}
is well-defined and coincides with the inverse map of $ s_{[g]} $.

We show (3). We denote by $ \phi $ the natural action of $ G $ on $ G/K $, which is transitive. One can see that
\begin{align*}
\phi(G) \subset \Aut(G/K,s),
\end{align*}
which yields that $ (G/K,s) $ is homogeneous.
\end{proof}

The quandle constructed from a quandle triplet $ (G,K,\sigma) $ in the way of Proposition~\ref{Prop3.2} is denoted by $ Q(G,K,\sigma) $.

We next construct a quandle triplet from a homogeneous quandle. For an action of a group $ G $ on a set $ X $, we denote the isotropy subgroup at $ x \in X $ by
\begin{align*}
G_{x} := \{ g \in G \mid g.x = x \}.
\end{align*}
Here we denote by $ g.x $ the action of $ g \in G $ on $ x \in X $.

\begin{Prop}\label{Prop3.3}
Let $ (X,s) $ be a quandle, $ G $ be a subgroup of $ \Aut(X,s) $, and $ x \in X $. Assume that $ s_{x} G s^{-1}_{x} \subset G $ holds, that is, 
\begin{align*}
\sigma : G \rightarrow G : g \mapsto s_{x} \circ g \circ s^{-1}_{x}
\end{align*}
is well-defined. Then we have
\begin{itemize}
\item[(1)]
$ (G,G_{x},\sigma) $ is a quandle triplet.
\item[(2)]
If $ G $ acts on $ X $ transitively, then $ Q(G,G_{x},\sigma) \simeq (X,s) $.
\end{itemize}
\end{Prop}

\begin{proof}
Let $ (X,s) $ be a quandle. We show (1). It is clear that $ \sigma \in \Aut(G) $. Take any $ g \in G_{x} $. Since $ g \in \Aut(X,s) $ and $ g.x = x $, we have
\begin{align*}
g \circ s_{x} = s_{g.x} \circ g = s_{x} \circ g.
\end{align*}
Then one has
\begin{align*}
\sigma(g) = s_{x} \circ g \circ s^{-1}_{x} = g \circ s_{x} \circ s^{-1}_{x} = g.
\end{align*}
This shows that $ g \in \Fix(\sigma,G) $, which completes the proof of (1).

We show (2). One knows that $ (G,G_{x},\sigma) $ is a quandle triplet from (1). Denote by $ s^{\prime} $ the quandle structure of $ Q(G,G_{x},\sigma) $, that is, $ (G/G_{x}, s^{\prime}) = Q(G,G_{x},\sigma) $. Consider the mapping
\begin{align*}
f : G/G_{x} \rightarrow X : [g] \mapsto g.x.
\end{align*}
We show that $ f $ is an isomorphism. Since the action of $ G $ on $ X $ is transitive, one knows that $ f $ is bijective. In order to show that $ f $ is a homomorphism, take any $ [g],[h] \in G/G_{x} $. Since $ g \in G \subset \Aut(X,s) $, we have $ g \circ s_{x} \circ g^{-1} = s_{g.x} $. This yields that
\begin{align*}
f \circ s^{\prime}_{[g]}([h]) &= (g \sigma(g^{-1}h)).x = g \circ s_{x} \circ g^{-1} h \circ s^{-1}_{x}(x) = s_{g.x} (h.x),\\
s_{f([g])} \circ f([h]) &= s_{g.x} (h.x).
\end{align*}
This shows that $ f $ is a homomorphism, which completes the proof.
\end{proof}

In the case of $ G = \Aut(X,s) $, the above proposition has been known in \cite{Joyce}. We gave a slight generalization, since $ G $ will be taken as the group of displacements in the latter argument.

We here summarize the correspondence between quandle triplets and homogeneous quandles. Note that the following is a direct corollary of Propositions~\ref{Prop3.2} and \ref{Prop3.3}

\begin{Thm}[\cite{Joyce}]\label{Thm3.4}
If $ (G,K,\sigma) $ is a quandle triplet, then $ Q(G,K,\sigma) $ is a homogeneous quandle. Conversely, all homogeneous quandles can be constructed in this way up to isomorphism.
\end{Thm}

\subsection{Splitting of quandle triplets}

In this subsection, we study the splitting of quandle triplets. Throughout this subsection, let $ G_{i} $ be a group, $ K_{i} $ be a subgroup of $ G_{i} $, and $ \sigma_{i} \in \Map(G_{i},G_{i}) $ ($ i =1,2 $).

\begin{Lem}\label{Lem3.5}
Assume that $ \sigma_{1} \times \sigma_{2} : G_{1} \times G_{2} \rightarrow G_{1} \times G_{2} $ is a group automorphism. Then $ \sigma_{1} $ and $ \sigma_{2} $ are also group automorphisms.
\end{Lem}

The proof of this lemma is an exercise of group theory. This describes a splitting of a group automorphism. Next we see a splitting of a quandle triplet.

\begin{Lem}\label{Lem3.6}
Assume that $ (G_{1} \times G_{2}, K_{1} \times K_{2}, \sigma_{1} \times \sigma_{2}) $ is a quandle triplet. Then $ (G_{1},K_{1},\sigma_{1}) $ and $ (G_{2},K_{2},\sigma_{2}) $ are also quandle triplets.
\end{Lem}

\begin{proof}
One knows that $ \sigma_{1} $ and $ \sigma_{2} $ are group automorphisms by Lemma~\ref{Lem3.5}. Therefore, we have only to show that
\begin{align*}
K_{1} \subset \Fix(\sigma_{1}, G_{1}), \quad K_{2} \subset \Fix(\sigma_{2}, G_{2}).
\end{align*}
Since $ (G_{1} \times G_{2}, K_{1} \times K_{2}, \sigma_{1} \times \sigma_{2}) $ is a quandle triplet, one has
\begin{align*}
K_{1} \times K_{2} \subset \Fix(\sigma_{1} \times \sigma_{2},G_{1} \times G_{2}).
\end{align*}
Furthermore, one can directly see that 
\begin{align*}
\Fix(\sigma_{1} \times \sigma_{2},G_{1} \times G_{2}) = \Fix(\sigma_{1},G_{1}) \times \Fix(\sigma_{2},G_{2}),
\end{align*}
which completes the proof.
\end{proof}

Finally in this subsection, we describe a splitting of the quandle constructed from a quandle triplet.

\begin{Prop}\label{Prop3.7}
Assume that $ (G_{1} \times G_{2}, K_{1} \times K_{2}, \sigma_{1} \times \sigma_{2}) $ is a quandle triplet. Then we have
\begin{align*}
Q(G_{1} \times G_{2}, K_{1} \times K_{2}, \sigma_{1} \times \sigma_{2}) \simeq Q(G_{1},K_{1},\sigma_{1}) \times Q(G_{2},K_{2},\sigma_{2}).
\end{align*}
\end{Prop}

\begin{proof}
Note that $ (G_{1},K_{1},\sigma_{1}) $ and $ (G_{2},K_{2},\sigma_{2}) $ are quandle triplets by Lemma~\ref{Lem3.6}. Define the quandle structures $ s, s^{\prime} $, and $ s^{\prime\prime} $ as follows:
\begin{align*}
((G_{1} \times G_{2})/(K_{1} \times K_{2}),s) &:= Q(G_{1} \times G_{2}, K_{1} \times K_{2}, \sigma_{1} \times \sigma_{2}),\\
(G_{1}/K_{1},s^{\prime}) &:= Q(G_{1},K_{1},\sigma_{1}),\\
(G_{2}/K_{2},s^{\prime\prime}) &:= Q(G_{2},K_{2},\sigma_{2}).
\end{align*}
We show that the following $ f $ is a quandle isomorphism:
\begin{align*}
f: (G_{1} \times G_{2})/(K_{1} \times K_{2}) \rightarrow G_{1}/K_{1} \times G_{2}/K_{2} : [(g_{1}, g_{2})] \mapsto ([g_{1}], [g_{2}]).
\end{align*}
It is clear that $ f $ is well-defined and bijective. We show that $ f $ is a homomorphism. Take any $ [(g_{1},g_{2})], [(h_{1},h_{2})] \in (G_{1} \times G_{2})/(K_{1} \times K_{2}) $. Then we have
\begin{align*}
f \circ s_{[(g_{1},g_{2})]}([(h_{1},h_{2})]) &= f ([(g_{1},g_{2}) (\sigma_{1} \times \sigma_{2}) ((g_{1},g_{2})^{-1}(h_{1},h_{2}))])\\ 
&= f([(g_{1}\sigma_{1}(g^{-1}_{1}h_{1}), g_{2}\sigma_{2}(g^{-1}_{2}h_{2}))])\\
&= ([g_{1}\sigma_{1}(g^{-1}_{1}h_{1})], [g_{2}\sigma_{2}(g^{-1}_{2}h_{2})]),\\
(s^{\prime} \times s^{\prime\prime})_{f([(g_{1},g_{2})])} \circ f([(h_{1},h_{2})]) &=  (s^{\prime} \times s^{\prime\prime})_{([g_{1}],[g_{2}])} ([h_{1}],[h_{2}])\\
&= ([g_{1}\sigma_{1}(g^{-1}_{1}h_{1})], [g_{2}\sigma_{2}(g^{-1}_{2}h_{2})]).
\end{align*}
This show that $ f $ is a homomorphism, which completes the proof.
\end{proof}

\subsection{The quandle triplets of the direct products of dihedral quandles}\label{subse3.3}

In this subsection, we study the quandle triplets which correspond to the direct products of dihedral quandles.

\begin{Lem}\label{Lem3.8}
$ (\Z_{n}, \{0\}, -\id_{\Z_{n}}) $ is a quandle triplet, and $ Q(\Z_{n},\{0\},-\id_{\Z_{n}}) $ is isomorphic to the dihedral quandle $ R_{n} $ of order $ n $.
\end{Lem}

\begin{proof}
One can easily see that $ (\Z_{n}, \{0\}, -\id_{\Z_{n}}) $ is a quandle triplet. Denote by $ (\Z_{n},s) := R_{n} $ and $ (\Z_{n}/\{0\},s^{\prime}) := Q(\Z_{n},\{0\},-\id_{\Z_{n}}) $. Consider the map
\begin{align*}
f : \Z_{n} \rightarrow \Z_{n}/\{0\} : x \rightarrow [x].
\end{align*}
We show that $ f $ is an isomorphism. It is clear that $ f $ is bijective. In order to show that $ f $ is a homomorphism, take any $ x,y \in \Z_{n} $. Since $ \sigma = -\id_{\Z_{n}} $, we have
\begin{align*}
f \circ s_{x}(y) &= f (2x-y) = [2x-y],\\
s^{\prime}_{f(x)} \circ f(y) &= s^{\prime}_{[x]}([y]) = [x \sigma(-x+y)] = [x-(-x+y)] = [2x-y].
\end{align*}
This shows that $ f $ is a homomorphism, which completes the proof.
\end{proof}

\begin{Prop}\label{Prop3.9}
Let $ q_{1}, \ldots, q_{n} \in \Z_{>0} $. Then $ (\Z_{q_{1}} \times \cdots \times \Z_{q_{n}}, \{ 0 \}, -\id_{\Z_{q_{1}} \times \cdots \times \Z_{q_{n}}}) $ is a quandle triplet, and we have
\begin{align*}
Q(\Z_{q_{1}} \times \cdots \times \Z_{q_{n}}, \{ 0 \}, -\id_{\Z_{q_{1}} \times \cdots \times \Z_{q_{n}}}) \simeq R_{q_{1}} \times \cdots \times R_{q_{n}}.
\end{align*}
\end{Prop}

\begin{proof}
One can easily see the first assertion. Furthermore, we know
\begin{align*}
-\id_{\Z_{q_{1}} \times \cdots \times \Z_{q_{n}}} = -\id_{\Z_{q_{1}}} \times \cdots \times -\id_{\Z_{q_{n}}}.
\end{align*}
Hence, from Proposition~\ref{Prop3.7} and Lemma~\ref{Lem3.8}, we have
\begin{align*}
&Q(\Z_{q_{1}} \times \cdots \times \Z_{q_{n}}, \{ 0 \}, -\id_{\Z_{q_{1}}} \times \cdots \times -\id_{\Z_{q_{n}}})\\
&\simeq Q(\Z_{q_{1}}, \{ 0 \}, -\id_{\Z_{q_{1}}}) \times \cdots \times Q(\Z_{q_{n}}, \{ 0 \}, -\id_{\Z_{q_{n}}})\\
&\simeq R_{q_{1}} \times \cdots \times R_{q_{n}}.
\end{align*}
This completes the proof.
\end{proof}

By this proposition, one can observe the following, which is relevant to our main theorem.

\begin{Rem}
Suppose that $ m $ and $ n $ are coprime to each other. Then one has a group isomorphism $ \Z_{mn} \simeq \Z_{m} \times \Z_{n} $. Hence Proposition~\ref{Prop3.9} yields that
\begin{align*}
R_{mn} \simeq R_{m} \times R_{n}.
\end{align*}
Therefore, every dihedral quandle can be expressed as a direct product of dihedral quandles with prime power cardinalities.
\end{Rem}

\section{Flat quandles}\label{sec4}

In this section, we define two new notions. One is the group of displacements of a quandle, and another is a flat quandle.

\subsection{The groups of displacements of quandles}

In this subsection, we define the groups of displacements of quandles, and study their properties.

\begin{Def}\label{Def4.1}
Let $ (X,s) $ be a quandle. The group generated by
\begin{align*}
\{ s_{x} \circ s_{y} \mid x,y \in X \}
\end{align*}
is called the \textit{group of displacements} of $ (X,s) $, and denoted by $ G^{0}(X,s) $.
\end{Def}

Note that, for any $ x,y \in X $, one has
\begin{align*}
s_{x} \circ s^{-1}_{y} = s_{x} \circ s_{x} \circ s^{-1}_{x} \circ s^{-1}_{y} = (s_{x} \circ s_{x}) \circ (s_{y} \circ s_{x})^{-1} \in G^{0}(X,s).
\end{align*}
By definition, one knows $ G^{0}(X,s) \subset \Inn(X,s) $. Therefore, if $ G^{0}(X,s) $ acts transitively on $ X $, then $ (X,s) $ is connected. In fact, the converse also holds.

\begin{Lem}\label{Lem4.2}
A quandle $ (X,s) $ is connected if and only if $ G^{0}(X,s) $ acts transitively on $ X $.
\end{Lem}

\begin{proof}
We have only to show the ``only if'' part. Assume that $ (X,s) $ is connected. Take any $ x,y \in X $. Then there exists $ f \in \Inn(X,s) $ such that $ f(x) = y $. If $ f \notin G^{0}(X,s) $, then we have
\begin{align*}
f \circ s_{x} \in G^{0}(X,s),\ f \circ s_{x} (x) = f (x) = y.
\end{align*}
This shows that $ G^{0}(X,s) $ acts transitively on $ X $.
\end{proof}

Finally in this subsection, we study the group of displacements of a direct product quandle. A quandle $ (X,s) $ is said to be \textit{involutive} if $ s_{x}^{2} = \id_{X} $ for any $ x \in X $.

\begin{Prop}\label{Prop4.3}
Let $ (X,s^{X}) $ and $ (Y,s^{Y}) $ be quandles. Then one has
\begin{align*}
G^{0}((X,s^{X}) \times (Y,s^{Y})) \subset G^{0}(X,s^{X}) \times G^{0}(Y,s^{Y}).
\end{align*}
Moreover, the equality holds if both of $ (X,s^{X}) $ and $ (Y,s^{Y}) $ are involutive.
\end{Prop}

\begin{proof}
In a similar way as Proposition~\ref{Prop2.9} (1), one can show the first assertion. We show the second assertion. Assume that both of $ (X,s^{X}) $ and $ (Y,s^{Y}) $ are involutive. Note that $ G^{0}(X,s^{X}) \times G^{0}(Y,s^{Y}) $ is generated by
\begin{align*}
\{ (s^{X}_{x_{1}} \circ s^{X}_{x_{2}}) \times \id_{Y} \mid x_{1}, x_{2} \in X \} \cup \{ \id_{X} \times (s^{Y}_{y_{1}} \circ s^{Y}_{y_{2}}) \mid y_{1}, y_{2} \in Y \}.
\end{align*}
Take any $ x_{1}, x_{2} \in X $, and $ y \in Y $. Since $ (s^{Y}_{y})^{2} = \id_{Y} $, one has
\begin{align*}
(s^{X}_{x_{1}} \circ s^{X}_{x_{2}}) \times \id_{Y} &= (s^{X} \times s^{Y})_{(x_{1},y)} \circ (s^{X} \times s^{Y})_{(x_{2},y)}\\
&\in G^{0}((X,s^{X}) \times (Y,s^{Y})).
\end{align*}
Similarly, for any $ y_{1}, y_{2} \in Y $, one has
\begin{align*}
\id_{X} \times (s^{Y}_{y_{1}} \circ s^{Y}_{y_{2}}) \in G^{0}((X,s^{X}) \times (Y,s^{Y})).
\end{align*}
This proves the second assertion.
\end{proof}

\subsection{Definition of flat quandles}

In this subsection, we define flat quandles. We refer to \cite{Loos} for a general theory of Riemannian symmetric spaces.

\begin{Thm}[\cite{Loos}]\label{Thm4.4}
Let $ (M,g) $ be a connected Riemannian symmetric space, and $ s_{x} $ be the symmetry at $ x $ in $ M $. Then the following conditions are equivalent$:$
\begin{itemize}
\item[(1)]
$ (M,g) $ is flat (that is, the Riemannian curvature vanishes identically).
\item[(2)]
The group generated by $ \{ s_{x} \circ s_{y} \mid x,y \in M \} $ is commutative.
\end{itemize}
\end{Thm}

We here define flat quandles, by referring to this theorem.

\begin{Def}\label{Def4.5}
A quandle $ (X,s) $ is said to be \textit{flat} if the group of displacements $ G^{0}(X,s) $ is commutative.
\end{Def}

Here we give some examples of flat quandles.

\begin{Ex}\label{Ex4.6}
The trivial quandles and the dihedral quandles are flat.
\end{Ex}

\begin{proof}
Let $ (X,s) $ be the trivial quandle. Recall that $ s_{x} = \id_{X} $ for every $ x \in X $. Hence $ G^{0}(X,s) = \{ \id_{X} \} $ is commutative.

Recall that the dihedral quandle $ R_{n} = (\Z_{n},s) $ is given by $ s_{x}(y) = 2x-y $. Hence one can see that
\begin{align*}
s_{x} \circ s_{y}(z) = s_{x}(2y-z) = 2(x-y)+z.
\end{align*}
Let us denote by $ r_{x}(z) := 2x + z $. Then we have
\begin{align*}
G^{0}(R_{n}) = \{ r_{x} \mid x \in \Z_{n} \}.
\end{align*}
This is clearly commutative.
\end{proof}

Finally in this subsection, we study a relationship between the flatness and direct products of quandles.

\begin{Prop}\label{Prop4.7}
Let $ (X,s^{X}) $ and $ (Y,x^{Y}) $ be quandles. Then the direct product $ (X,s^{X}) \times (Y,s^{Y}) $ is flat if and only if both of $ (X,s^{X}) $ and $ (Y,s^{Y}) $ are flat.
\end{Prop}

\begin{proof}
The ``if'' part easily follows from Proposition~\ref{Prop4.3}. We show the ``only if'' part. Assume that the direct product $ (X,s^{X}) \times (Y,s^{Y}) $ is flat. Take any $ x_{1}, x_{2}, x_{3}, x_{4} \in X $, and $ y \in Y $. Then one knows 
\begin{align*}
( s^{X}_{x_{1}} \circ s^{X}_{x_{2}} ) \times ( s^{Y}_{y} \circ s^{Y}_{y} ), ( s^{X}_{x_{3}} \circ s^{X}_{x_{4}} ) \times ( s^{Y}_{y} \circ s^{Y}_{y} ) \in G^{0}((X,s^{X}) \times (Y,s^{Y})).
\end{align*}
By assumption, they are commutative. Therefore, $ s^{X}_{x_{1}} \circ s^{X}_{x_{2}} $ and $ s^{X}_{x_{3}} \circ s^{X}_{x_{4}} $ are commutative. This proves that $ G^{0}(X,s^{X}) $ is commutative, and hence $ (X,s^{X}) $ is flat. Similarly, $ (Y,s^{Y}) $ is flat.
\end{proof}

\section{Flat connected quandles}\label{sec5}

In this section, we study flat connected quandles (but not necessarily finite). We show that such quandles are constructed from particular kinds of quandle triplets.

First of all, we give a general lemma for quandle triplets of connected quandles. Namely, one can construct quandle triplets by the groups of displacements.

\begin{Lem}\label{Lem5.1}
Let $ (X,s) $ be a connected quandle and $ x \in X $. We put
\begin{align*}
G := G^{0}(X,s),\ K := G_{x},\ \sigma(g) := s_{x} \circ g \circ s^{-1}_{x} \quad (g \in G).
\end{align*}
Then $ (G,K,\sigma) $ is a quandle triplet, and $ (X,s) \simeq Q(G,K,\sigma) $.
\end{Lem}

\begin{proof}
Let $ (X,s) $ be a connected quandle and $ x \in X $. In view of Proposition~\ref{Prop3.3}, we have only to prove that $ \sigma(G) \subset G $ holds and $ G $ acts transitively on $ X $. Take any $ g \in G $. Then we have
\begin{align*}
\sigma(g) =  s_{x} \circ g \circ s^{-1}_{x} = g \circ (s_{g^{-1}.x} \circ s^{-1}_{x}) \in G.
\end{align*}
This shows $ \sigma(G) \subset G $. Furthermore, since $ (X,s) $ is connected, the action of $ G=G^{0}(X,s) $ on $ X $ is transitive from Lemma~\ref{Lem4.2}.
\end{proof}

We next show that, if $ (X,s) $ are flat connected quandles, then the quandle triplets defined above are of particular forms.

\begin{Prop}\label{Prop5.2}
Let $ (X,s) $ be a flat connected quandle, $ x \in X $, and $ (G,K,\sigma) $ be the quandle triplet defined in Lemma~\ref{Lem5.1}. Then we have
\begin{itemize}
\item[(1)]
$ G $ is commutative,
\item[(2)]
$ K = \{ \id_{X} \} $,
\item[(3)]
$ \sigma $ is involutive.
\end{itemize}
\end{Prop}

\begin{proof}
Since $ (X,s) $ is flat, one knows by definition that $ G = G^{0}(X,s) $ is commutative. This proves (1).

We show (2). Take any $ k \in K $ and $ y \in X $. Since the action of $ G $ on $ X $ is transitive, there exists $ g \in G $ such that $ g.x = y $. One has $ gk=kg $ since $ G $ is commutative, and hence we have
\begin{align*}
k.y = k.(g.x) = g.(k.x) = g.x = y.
\end{align*}
This yields that $ k = \id_{X} $, which proves (2).

We show (3). Take any $ g \in G $. Since $ s^{2}_{x} \in G $ and $ G $ is commutative, one has
\begin{align*}
s^{2}_{x} \circ g = g \circ s^{2}_{x}.
\end{align*}
This yields that
\begin{align*}
\sigma^{2}(g) = \sigma(s_{x} \circ g \circ s^{-1}_{x}) = s^{2}_{x} \circ g \circ s^{-2}_{x} = g \circ s^{2}_{x} \circ s^{-2}_{x} = g.
\end{align*}
This completes the proof.
\end{proof}

For a flat connected quandle, we have constructed a particular kind of quandle triplet $ (G,K,\sigma) $. In the remaining of this section, we study the converse. To be precise, we study the quandles $ Q(G,K,\sigma) $ with $ G $ commutative, $ K = \{ e \} $, and $ \sigma $ involutive. Note that such quandles $ Q(G,K,\sigma) $ might be disconnected, since the dihedral quandles with even cardinalities can be constructed in this way (see Subsection~\ref{subse3.3}). First of all, we show the following two lemmas.

\begin{Lem}\label{Lem5.3}
Let $ (G,K,\sigma) $ be a quandle triplet, and assume that $ \sigma $ is involutive. Then the quandle $ Q(G,K,\sigma) $ is involutive.
\end{Lem}

\begin{proof}
Denote by $ (X,s) := Q(G,K,\sigma) $, and take any $ g \in G $. We have only to show that $ s_{[g]}^{2} = \id_{X} $. For any $ h \in G $, one has
\begin{align*}
s_{[g]}^{2}([h]) = s_{[g]}([g \sigma(g^{-1}h)]) = [g\sigma(g^{-1}(g\sigma(g^{-1}h)))] = [g\sigma^{2}(g^{-1}h)] = [h],
\end{align*}
since $ \sigma^{2} = \id_{G} $. This completes the proof.
\end{proof}

\begin{Lem}\label{Lem5.4}
Let $ (G,K,\sigma) $ be a quandle triplet, and assume that $ G $ is commutative and $ \sigma $ is involutive. Denote by $ (X,s) := Q(G,K,\sigma) $. Then one has
\begin{align*}
s_{[g]} \circ s_{[h]} ([e]) = s_{[gh^{-1}]} ([e]) \quad (\forall g,h \in G).
\end{align*}
\end{Lem}

\begin{proof}
Take any $ g,h \in G $. Since $ \sigma $ is an involutive automorphism, we have
\begin{align*}
s_{[g]} \circ s_{[h]}([e]) &= s_{[g]} ([h \sigma(h^{-1})]) = [g \sigma(g^{-1}h \sigma(h^{-1}))] = [g \sigma(g^{-1}h)h^{-1}],\\
s_{[gh^{-1}]}([e]) &= [gh^{-1} \sigma(hg^{-1})].
\end{align*}
They coincide, since $ G $ is commutative. This completes the proof.
\end{proof}

We are now in position to study our quandles $ Q(G,K,\sigma) $. The following gives a criteria for our $ Q(G,K,\sigma) $ to be connected.

\begin{Prop}\label{Prop5.5}
Let $ (G,K,\sigma) $ be a quandle triplet, and assume that $ G $ is commutative, $ K = \{e\} $, and $ \sigma $ is involutive. We identify $ G \cong G/\{e\} $, and denote by $ (G,s) := Q(G,K,\sigma) $. Then $ (G,s) $ is connected if and only if the following $ \varphi $ is surjective$:$
\begin{align*}
\varphi : G \rightarrow G : g \mapsto s_{g}(e).
\end{align*}
\end{Prop}

\begin{proof}
In order to show this proposition, it is enough to prove
\begin{align*}
\Inn(G,s).e = \varphi(G).
\end{align*}
We prove this equality. The inclusion $ (\supset) $ is clear, since $ s_{g} \in \Inn(G,s) $. We show $ (\subset) $. Take any $ k \in \Inn(G,s) $. Since $ \sigma $ is involutive, Lemma~\ref{Lem5.3} yields that $ (G,s) $ is involutive. That is, $ s^{-1}_{g} = s_{g} $ holds for any $ g \in G $. Therefore, there exist $ g_{1}, \ldots, g_{n} \in G $ such that $ k = s_{g_{1}} \circ \cdots \circ s_{g_{n}} $. Furthermore, an easy induction in terms of Lemma~\ref{Lem5.4} shows that there exists $ g_{0} \in G $ such that
\begin{align*}
k.e = s_{g_{0}}(e).
\end{align*}
This completes the proof, since $ s_{g_{0}}(e) = \varphi(g_{0}) \in \varphi(G) $.
\end{proof}

Note that the map $ \varphi $ defined above is a group homomorphism. We will use this map $ \varphi $ in the next section.

\section{Flat connected finite quandles}\label{sec6}

In this section, we prove Theorem~\ref{Thm1.1}, which classifies flat connected finite quandles completely. First of all, we show the ``if'' part of Theorem~\ref{Thm1.1}. Recall that $ R_{n} $ denotes the dihedral quandle of order $ n $.

\begin{Prop}\label{Prop6.1}
Let $ q_{1}, \ldots, q_{n} $ be odd prime powers. Then the direct product $ R_{q_{1}} \times \cdots \times R_{q_{n}} $ is a flat connected finite quandle.
\end{Prop}

\begin{proof}
We show that $ X := R_{q_{1}} \times \cdots \times R_{q_{n}} $ is flat, connected, and finite. One knows that $ R_{q_{1}}, \ldots, R_{q_{n}} $ are flat from Example~\ref{Ex4.6}, and hence $ X $ is flat from Proposition~\ref{Prop4.7}. Furthermore, since $ q_{1}, \ldots, q_{n} $ are odd, $ R_{q_{1}}, \ldots, R_{q_{n}} $ are connected from Example~\ref{Ex2.7}, and hence $ X $ is connected from Proposition~\ref{Prop2.9} (2). In addition, $ X $ is obviously finite. This completes the proof.
\end{proof}

In the remaining of this section, we show the ``only if'' part of Theorem~\ref{Thm1.1}. For this purpose, we study properties of the corresponding quandle triplets.

\begin{Lem}\label{Lem6.2}
Let $ (X,s) $ be a flat, connected, and finite quandle, and $ (G,K,\sigma) $ be the quandle triplet defined in Lemma~\ref{Lem5.1}. Then we have
\begin{itemize}
\item[(1)]
$ G $ is finite,
\item[(2)]
$ \Fix(\sigma,G) = \{ e \} $,
\item[(3)]
$ \sigma(g) = g^{-1} $ for every $ g \in G $.
\end{itemize}
\end{Lem}

\begin{proof}
The assertion (1) is clear, since $ X $ is finite. We show (2). Take any $ g \in \Fix(\sigma,G) $, and consider the mapping $ \varphi : G \rightarrow G $ defined in Proposition~\ref{Prop5.5}. Then one has
\begin{align*}
\varphi(g) = s_{g}(e) = g \sigma (g^{-1}) = g g^{-1} = e = s_{e}(e) = \varphi(e).
\end{align*}
We here recall that $ \varphi $ is surjective from Proposition~\ref{Prop5.5}. Hence $ \varphi $ is injective, since $ G $ is finite. This shows that $ g = e $, which proves (2).

We show (3). Take any $ g \in G $. One knows from Proposition~\ref{Prop5.2} that $ \sigma $ is involutive and $ G $ is commutative. We thus have
\begin{align*}
\sigma(g \sigma(g)) = \sigma(g)g = g \sigma(g).
\end{align*}
This means $ g \sigma(g) \in \Fix(\sigma,G) $. It then follows from (2) that $ g \sigma(g) = e $, and hence $ \sigma(g) = g^{-1} $. This completes the proof.
\end{proof}

In order to study flat connected finite quandles, it is sufficient to study the above quandle triplets. The following proposition completes the proof of Theorem~\ref{Thm1.1}.

\begin{Prop}\label{Prop6.3}
Let $ (X,s) $ be a flat, connected, and finite quandle. Then there exist odd prime powers $ q_{1}, \ldots, q_{n} $ such that $ (X,s) \simeq R_{q_{1}} \times \cdots \times R_{q_{n}} $.
\end{Prop}

\begin{proof}
Let $ (G,K,\sigma) $ be the quandle triplet defined in Lemma~\ref{Lem5.1}. Note that this satisfies the assumptions in Proposition~\ref{Prop5.2} and Lemma~\ref{Lem6.2}. In particular, $ G $ is finite and commutative. Then the fundamental theorem of finite commutative groups yields that there exist prime powers $ q_{1}, \ldots, q_{n} $ such that
\begin{align*}
G \simeq \Z_{q_{1}} \times \cdots \times \Z_{q_{n}},
\end{align*}
which is an isomorphism as groups. Hence, the inverse element of $ g \in G $ can be written as $ -g $. It then follows from Lemma~\ref{Lem6.2} (3) that
\begin{align*}
\sigma = -\id_{\Z_{q_{1}} \times \cdots \times \Z_{q_{n}}}.
\end{align*}
Therefore, Lemma~\ref{Lem5.1} and Proposition~\ref{Prop3.9} show that
\begin{align*}
(X,s) \simeq Q(\Z_{q_{1}} \times \cdots \times \Z_{q_{n}}, \{ 0 \}, -\id_{\Z_{q_{1}} \times \cdots \times \Z_{q_{n}}}) \simeq R_{q_{1}} \times \cdots \times R_{q_{n}}.
\end{align*}
It remains to show that $ q_{1}, \ldots, q_{n} $ are odd. Since $ (X,s) $ is connected, Proposition~\ref{Prop2.9} yields that all of $ R_{q_{1}}, \ldots, R_{q_{n}} $ are connected. Therefore, from Example~\ref{Ex2.7}, we conclude that $ q_{1}, \ldots, q_{n} $ are odd.
\end{proof}

\end{document}